\documentclass[11pt]{amsart}

\usepackage{amsmath}
\usepackage{amsthm}
\usepackage{amssymb}
\usepackage{hyperref}
\hypersetup{colorlinks=true, linkcolor=blue}

\newtheoremstyle{theoremstyle}{\baselineskip}{}{\itshape}{}{\bfseries}{.}{5pt}{\thmnumber{#2\ \,}\thmname{#1}\thmnote{ \mdseries #3}}
\theoremstyle{theoremstyle}
\newtheorem{theorem}{Theorem}[section]
\newtheorem{lemma}[theorem]{Lemma}
\newtheorem{corollary}[theorem]{Corollary}

\newtheoremstyle{examplestyle}{\baselineskip}{}{\normalfont}{}{\bfseries}{.}{5pt}{\thmnumber{#2\ \,}\thmname{\mdseries\itshape #1}}
\theoremstyle{examplestyle}

\newtheorem{remark}[theorem]{Remark}

\numberwithin{equation}{theorem}

\newcommand{\id}{\operatorname{id}}
\newcommand{\sgn}{\operatorname{sgn}}

\newcommand{\N}{\mathbf N}
\newcommand{\Z}{\mathbf Z}

\usepackage{color}

\title{A recursion formula for the irreducible characters of the symmetric group}
\author{Randall R. Holmes}
\subjclass[2010]{20C30, 20C15, 20C40}
\keywords{symmetric group, character, recursion}

\address{Randall R. Holmes,
Department of Mathematics and Statistics,
Auburn University,
Auburn AL,
36849,
USA,
\texttt{holmerr@auburn.edu}}

\begin{document}
\setcounter{section}{-1}
\begin{abstract}
The branching theorem expresses irreducible character values for the symmetric group \(S_n\) in terms of those for \(S_{n-1}\), but it gives the values only at elements of \(S_n\) having a fixed point.  We extend the theorem by providing a recursion formula that handles the remaining cases.  It expresses these character values in terms of values for \(S_{n-1}\) together with values for \(S_n\) that are already known in the recursive process.  This provides an alternative to the Murnaghan-Nakayama formula.
\end{abstract}

\maketitle
\markright{RECURSION FORMULA FOR IRREDUCIBLE CHARACTERS}

\section{Introduction}
\label{sec:Introduction}

Let \(n\) be a nonnegative integer and denote by \(S_n\) the symmetric group of degree \(n\).  For a partition \(\alpha\) of \(n\) (written \(\alpha\vdash n\)), denote by \(\zeta^\alpha\) the corresponding irreducible character of \(S_n\), and denote by \(\zeta^\alpha_\beta\) the value of \(\zeta^\alpha\) at the conjugacy class of \(S_n\) corresponding to \(\beta\vdash n\).

Assume that \(n>0\).  Let \(\alpha,\beta\vdash n\) and let \(\beta_m\) be the last (nonzero) part of \(\beta\).  If \(\beta_m=1\), then the conjugacy class of \(S_n\) corresponding to \(\beta\) contains a permutation that fixes \(n\), so the branching theorem \cite[2.4.3, p.~59]{MR644144} expresses \(\zeta^\alpha_\beta\) in terms of irreducible character values for \(S_{n-1}\).  The main result of this paper (Theorem \ref{thm:RecursionFormula}) is the following recursion formula for \(\zeta^\alpha_\beta\) in the remaining case \(\beta_m\not=1\):
\[
\label{eqn:RecursionFormulaSpecialCase}
\tag{1}
\zeta^\alpha_\beta=(\beta_m-1)^{-1}\left[\sum_{\substack{i=1\\\alpha_i>\alpha_{i+1}\\\,}}^{l(\alpha)}(\alpha_i-i)\zeta^{\alpha-\varepsilon_i}_{\beta-\varepsilon_m}-\sum_{\substack{j=1\\\beta_j<\beta_{j-1}}}^{m-1}\mu_j\beta_j\zeta^\alpha_{\beta+\varepsilon_j-\varepsilon_m}
\right].
\]
Here, \(l(\alpha)\) is the length of \(\alpha\), \(\varepsilon_i\) is the sequence with \(1\) in the \(i\)th position and zeros elsewhere, \(\mu_j\) is the multiplicity of \(\beta_j\) in the partition \(\beta-\varepsilon_m\), and \(\beta_0:=\infty\).

For each \(i\) in the first sum, \(\alpha-\varepsilon_i\) is a partition of \(n-1\), so \(\zeta^{\alpha-\varepsilon_i}\) is an irreducible character of \(S_{n-1}\).  For each \(j\) in the second sum, \(\beta+\varepsilon_j-\varepsilon_m\) is less than \(\beta\) with respect to the reverse lexicographical ordering of partitions.  It follows that the character values can be found recursively using the formula.

The character tables for \(S_n\), \(2\le n\le10\), generated (with the aid of a computer) using the recursion formula (\ref*{eqn:RecursionFormulaSpecialCase}), together with the branching theorem, are in agreement with the character tables appearing in \cite[Appendix I]{MR644144}.

The Murnaghan-Nakayama formula \cite[2.4.7, p.~60]{MR644144} also expresses the irreducible character values \(\zeta^\alpha_\beta\) recursively.  It does so by expressing such a value for \(S_n\) in terms of character values for \(S_{n-r}\), where \(r\) is a (nonzero) part of the partition \(\beta\).    The computation involves the removal of the various ``rims'' of length \(r\) from the Young diagram of the partition \(\alpha\).  The recursion formula (\ref*{eqn:RecursionFormulaSpecialCase}) may allow for proofs and computations in situations where the Murnaghan-Nakayama formula cannot be easily applied.

\section{General notation and background}
\label{sec:GeneralNotationAndBackground}

Let \(n\in\N:=\{0,1,2,\dots\}\).  Denote by \(S_n\) the symmetric group on the set \(\mathbf n:=\{1,2,\dots,n\}\) and by \(\varepsilon\) the identity element of \(S_n\).

For \(l\in\N\), put
\[
\Gamma_l=\{\gamma=(\gamma_1,\gamma_2,\dots,\gamma_l,0,0,\dots\,)\mid \gamma_i\in\Z\}
\]
and \(\Gamma^+_l=\{\gamma\in\Gamma_l\mid\gamma_i\ge0\ \forall\,i\}\).  Further, put \(\Gamma=\bigcup_l\Gamma_l\) and \(\Gamma^+=\bigcup_l\Gamma^+_l\).

For \(\gamma\in\Gamma\), denote by \(l(\gamma)\) (length of \(\gamma\)) the least \(l\in\N\) for which \(\gamma\in\Gamma_l\).  If \(\gamma\in\Gamma\) is nonzero, then \(l(\gamma)\) is the least \(l\in\Z^+\) for which \(\gamma_l\not=0\).

An element \(\gamma\) of \(\Gamma^+\) is a \emph{partition} of \(n\), written \(\gamma\vdash n\), if \(|\gamma|:=\sum_i\gamma_i=n\) and \(\gamma_i\ge\gamma_{i+1}\) for each \(i\).

The \emph{cycle structure} of a permutation \(\sigma\in S_n\) is the partition of \(n\) obtained by writing the lengths of the cycles in a disjoint cycle decomposition of \(\sigma\) in nonincreasing order (followed by zeros).  Two elements of \(S_n\) are conjugate if and only if they have the same cycle structure \cite[p.~9]{MR644144}.

Let \(\beta\in\Gamma^+\) with \(|\beta|=n\). For each \(j\in\Z^+\), put \(\sigma^\beta_j=(s_{j1},s_{j2},\dots,s_{j\beta_j})\in S_n\), where \(s_{ji}=i+\sum_{k=1}^{j-1}\beta_k\). (If \(\beta_j=0\), then \(\sigma^\beta_j=(\,)=\varepsilon\).)  Put \(\sigma_\beta=\prod_j\sigma^\beta_j\). For instance, if \(\beta=(4,2,2,1,0,0,\dots\,)\), then
\[
\sigma_\beta=(1,2,3,4)(5,6)(7,8)(9)\in S_9.
\]
It follows from the preceding paragraph that \(\{\sigma_\beta\mid\beta\vdash n\}\) is a complete set of representatives of the conjugacy classes of \(S_n\).

Let \(\alpha\in\Gamma\) with \(|\alpha|=n\).  Denote by \(T^\alpha\) the set of all sequences \(t=(t_1,t_2,\dots\,)\) with the \(t_i\) pairwise disjoint subsets of \(\mathbf n\) such that \(|t_i|=\alpha_i\) for each \(i\) (and note that this final condition implies that \(T^\alpha=\emptyset\) if \(\alpha\notin\Gamma^+\)).  An element \(t\) of \(T^\alpha\) is called an \emph{\(\alpha\)-tabloid}; it can be regarded as an \(\alpha\)-tableau with unordered rows (cf. \cite[p.~41]{MR644144}).

Assume that \(\alpha\in\Gamma^+\). Put \(t^\alpha=(\dot\sigma^\alpha_1,\dot\sigma^\alpha_2,\dots\,)\in T^\alpha\), where we use the notation \(\dot\sigma:=\{s_1,s_2,\dots,s_m\}\) for a cycle \(\sigma=(s_1,s_2,\dots,s_m)\) in \(S_n\).  For instance, if \(\alpha=(3,4,2,0,0,\dots)\), then
\[
t^\alpha=(\{1,2,3\},\{4,5,6,7\},\{8,9\},\emptyset,\emptyset,\dots\,)\in T^\alpha.
\]

An action of the symmetric group \(S_n\) on the set \(T^\alpha\) is given by \(\sigma t=(\sigma(t_1),\sigma(t_2),\dots\,)\) (\(\sigma\in S_n\), \(t\in T^\alpha\)). The \emph{Young subgroup} \(S_\alpha\) of \(S_n\) corresponding to \(\alpha\) is given by
\[
S_\alpha=\{\sigma\in S_n\mid \sigma t^\alpha=t^\alpha\}.
\]

\section{Reciprocity for permutation characters}
\label{sec:ReciprocityForPermutationCharacters}

For \(\alpha\in\Gamma\), define \(\xi^\alpha:S_{|\alpha|}\to\Z\) by
\[
\xi^\alpha=
\begin{cases}
1_{S_\alpha}\uparrow S_{|\alpha|},&\text{if \(\alpha\in\Gamma^+\)},\\
0,&\text{otherwise}.
\end{cases}
\]
Here \(1_{S_\alpha}\uparrow S_{|\alpha|}\) denotes the character of \(S_{|\alpha|}\) induced from the trivial character of the Young subgroup \(S_\alpha\) of \(S_{|\alpha|}\); it is a permutation character.  (See \cite{MR0450380} for general character theory.)

For \(\alpha\in\Gamma\) and \(\beta\in\Gamma^+\), put
\[
\xi^\alpha_\beta=
\begin{cases}
  \xi^\alpha(\sigma_\beta),&\text{if \(|\alpha|=|\beta|\)},\\
  0,&\text{otherwise}.
\end{cases}
\]

The goal in this section is to establish Theorem \ref{thm:ReciprocalPermutation} below, which is a reciprocal relationship involving these values \(\xi^\alpha_\beta\).

For \(\alpha\in\Gamma\) and \(\sigma\in S_{|\alpha|}\), put
\[
T^\alpha_\sigma=\{t\in T^\alpha\mid \sigma t=t\}.
\]
The following result was observed in \cite[p.~41]{MR644144}.  We provide a proof for the convenience of the reader.

\begin{lemma}
\label{lemma:PermutationCharacterEqualsCardinality}
  Let \(\alpha\in\Gamma\) and put \(n=|\alpha|\).  For every \(\sigma\in S_n\), we have \(\xi^\alpha(\sigma)=|T^\alpha_\sigma|\).
\end{lemma}
\begin{proof}
  Let \(\sigma\in S_n\).  If \(\alpha\notin\Gamma^+\), then \(T^\alpha\) (and hence \(T^\alpha_\sigma\)) is empty and the equality holds.  So, without loss of generality, we assume that \(\alpha\in\Gamma^+\).  The character \(\xi^\alpha\) is afforded by the permutation representation \(\rho\) of \(S_n\) corresponding to the \(S_n\)-set \(C=\{\tau S_\alpha\mid \tau\in S_n\}\) with action given by left multiplication, so the character value \(\xi^\alpha(\sigma)\), which is the trace of \(\rho(\sigma)\), equals the cardinality of the set \(\{c\in C\mid\sigma c=c\}\).  Now the \(S_n\)-set \(C\) is isomorphic to the \(S_n\)-set \(S_nt^\alpha=T^\alpha\) via \(\tau S_\alpha\mapsto\tau t^\alpha\), so \(\xi^\alpha(\sigma)=|\{t\in T^\alpha\mid\sigma t=t\}|=|T^\alpha_\sigma|.\)
\end{proof}

Let \(\alpha\in\Gamma^+\) and \(t\in T^\alpha\).  Let \(\sigma\in S_{|\alpha|}\) and write \(\sigma=\prod_{j=1}^m\sigma_j\) with the \(\sigma_j\) disjoint cycles.

\begin{lemma}
\label{lemma:FixesTabloidIff}
The following are equivalent:
\begin{itemize}
  \item[(i)] \(\sigma t=t\);
  \item[(ii)] for each \(1\le j\le m\), we have \(\dot\sigma_j\subseteq t_i\) for some \(i\).
\end{itemize}
\end{lemma}
\begin{proof}
  Assume that (i) holds. Let \(1\le j\le m\).  Due to the disjointness of the cycles, we have \(\sigma_j t= t\).  Let \(k\in\dot\sigma_j\).  We have \(k\in t_i\) for some \(i\), and for every integer \(l\) we have \({\sigma_j}^l(k)\in{\sigma_j}^l(t_i)=t_i\), so \(\dot\sigma_j\subseteq t_i\).  Hence (ii) holds.

  Now assume that (ii) holds.  Let \(1\le j\le m\).  We have \(\dot\sigma_j\subseteq t_i\) for some \(i\).  Then \(\sigma_j (t_i)=t_i\) and for \(k\not=i\), \(\dot\sigma_j\cap t_k=\emptyset\) so \(\sigma_j(t_k)=t_k\) as well.  Therefore, \(\sigma_j t=t\).  It follows that (i) holds.
\end{proof}

The set \(\Gamma\) is a group under componentwise addition.  For \(i\in\Z^+\), put \(\varepsilon_i=(0,\dots,0,\underset i1,0,\dots\,)\in\Gamma\).

\begin{theorem}
\label{thm:ReciprocalPermutation}
  Let \(l,m\in\N\).  For every \(\alpha\in\Gamma_l\) and \(\beta\in\Gamma^+_m\), we have
  \[
  \sum_{i=1}^l(\alpha_i-1)\xi^{\alpha-\varepsilon_i}_\beta=\sum_{j=1}^m\beta_j\xi^\alpha_{\beta+\varepsilon_j}.
  \]
\end{theorem}
\begin{proof}
  Before fixing elements, we define in general
  \[
    T^\alpha_\sigma(i,X):=\{t\in T^\alpha_\sigma\mid t_i=X\}
  \]
  for \(\alpha\in\Gamma\), \(\sigma\in S_{|\alpha|}\), \(i\in\Z^+\), and a set \(X\).

  Let \(\alpha\in\Gamma_l\) and \(\beta\in\Gamma^+_m\) and put \(n=|\alpha|\).  If either \(\alpha\notin\Gamma^+\) or \(|\beta|\not=n-1\), then \(\xi^{\alpha-\varepsilon_i}_\beta,\xi^\alpha_{\beta+\varepsilon_j}=0\) for every \(i\) and \(j\), and the equality holds.  Also, for \(i>l(\alpha)\), we have \(\alpha-\varepsilon_i\notin\Gamma^+\) so \(\xi^{\alpha-\varepsilon_i}=0\) and the corresponding term on the left is zero, and for \(j>l(\beta)\), we have \(\beta_j=0\) and the corresponding term on the right is zero.   So, without loss of generality, we assume that \(\alpha\in\Gamma^+\), \(|\beta|=n-1\), \(l=l(\alpha)\), and \(m=l(\beta)\).

  For each \(j\in\Z^+\), put \(\sigma_j=\sigma^\beta_j\), and put \(\sigma=\sigma_\beta\) (see Section \ref{sec:GeneralNotationAndBackground}).  We have \(\sigma=\prod_j\sigma_j\), a product of disjoint cycles with each \(\sigma_j\) of length \(\beta_j\).

  Let \(X\) be a set and put
  \[
  J(X)=\{j\in\Z^+\mid\dot\sigma_j\subseteq X\}.
  \]
  Fix \(i\in\Z^+\).

  Step 1:\quad\emph{If \(T^{\alpha-\varepsilon_i}_\sigma(i,X)\) is nonempty, then \(\alpha_i-1=\sum_{j\in J(X)}\beta_j\).}

  Assume that \(T^{\alpha-\varepsilon_i}_\sigma(i,X)\) is nonempty.  Then there exists \(t\in T^{\alpha-\varepsilon_i}\) such that \(\sigma t=t\) and \(t_i=X\).
  We claim that \(X=\bigcup_{j\in J(X)}\dot\sigma_j\).  Let \(k\in X=t_i\).  We have \(1\le k\le n-1\), so \(k\in\dot\sigma_j\cap t_i\) for some \(1\le j\le m\).  Then Lemma \ref{lemma:FixesTabloidIff} implies \(\dot\sigma_j\subseteq t_i=X\), so that \(j\in J(X)\).  This gives one inclusion of the claim, and the other inclusion is immediate.  Therefore,
  \[
  \alpha_i-1=|X|=\sum_{j\in J(X)}|\dot\sigma_j|=\sum_{j\in J(X)}\beta_j
  \]
  and Step 1 is complete.

  Fix \(j\in\Z^+\) and denote by \(\sigma[j]\) the permutation in \(S_n\) obtained from \(\sigma\) by appending \(n\) to the \(j\)th cycle \(\sigma_j\).  More precisely, if we write \(\sigma_j=(s_{j1},s_{j2},\dots,s_{j\beta_j})\), then \(\sigma[j]={\sigma_j}'\prod_{k\not=j}\sigma_k\), where \({\sigma_j}'=(s_{j1},s_{j2},\dots,s_{j\beta_j},n)\).

  Step 2:\quad \emph{For each \(X\subseteq\mathbf n':=\{1,2,\dots,n-1\}\), we have
  \[
  |T^\alpha_{\sigma[j]}(i,X\cup\{n\})|=
  \begin{cases}
    |T^{\alpha-\varepsilon_i}_\sigma(i,X)|,&\text{if \(j\in J(X)\),}\\
    0,&\text{otherwise}.
  \end{cases}
  \]
  }

  Let \(X\subseteq\mathbf n'\). If \(\alpha-\varepsilon_i\notin\Gamma^+\), then \(\alpha_i=0\), so both \(T^\alpha_{\sigma[j]}(i,X\cup\{n\})\) and \(T^{\alpha-\varepsilon_i}_\sigma(i,X)\) are empty and the statement holds.  So, without loss of generality, we assume that \(\alpha-\varepsilon_i\in\Gamma^+\).

  Assume that \(j\in J(X)\), so that \(\dot\sigma_j\subseteq X\). For an \((\alpha-\varepsilon_i)\)-tabloid \(t\), denote by \(t^+\) the \(\alpha\)-tabloid with \((t^+)_i=t_i\cup\{n\}\) and \((t^+)_k=t_k\), \(k\not=i\).  It follows from Lemma \ref{lemma:FixesTabloidIff} that \(t\mapsto t^+\) defines a bijection \(T^{\alpha-\varepsilon_i}_\sigma(i,X)\to T^\alpha_{\sigma[j]}(i,X\cup\{n\})\) with inverse given by \(t\mapsto t^-\) where \((t^-)_i=t_i\setminus\{n\}\) and \((t^-)_k=t_k\), \(k\not=i\).  Therefore, the first case follows.

  Now assume that \(j\notin J(X)\), so that \(\dot\sigma_j\nsubseteq X\). We have \(s_{jk}\notin X\) for some \(1\le k\le\beta_j\). Let \(t\) be an \(\alpha\)-tabloid with \(t_i=X\cup\{n\}\).  Since \(s_{jk}\not=n\), we have \(s_{jk}\notin t_i\).  Therefore, \(n\in\dot{\sigma_j}'\cap t_i\), but \(s_{jk}\in\dot{\sigma_j}'\setminus t_i\), so Lemma \ref{lemma:FixesTabloidIff} implies that \(\sigma[j] t\not=t\).  We conclude that \(T^\alpha_{\sigma[j]}(i,X\cup\{n\})\) is empty and the second case follows.  This completes Step 2.

  We are now ready to establish the equality in the theorem. Using Lemma \ref{lemma:PermutationCharacterEqualsCardinality} (and recalling that \(\sigma=\sigma_\beta\)) we have
  \begin{align*}
    \sum_{i=1}^l(\alpha_i-1)\xi^{\alpha-\varepsilon_i}_\beta&=\sum_{i=1}^l(\alpha_i-1)|T^{\alpha-\varepsilon_i}_\sigma|\\
    &=\sum_{i=1}^l\sum_{X\subseteq\mathbf n'}(\alpha_i-1)|T^{\alpha-\varepsilon_i}_\sigma(i,X)|.
  \end{align*}

  For each \(1\le i\le l\) and \(X\subseteq\mathbf n'\) we have, using Step 1 and then Step 2,
  \[
  (\alpha_i-1)|T^{\alpha-\varepsilon_i}_\sigma(i,X)|
  =\sum_{j\in J(X)}\beta_j|T^{\alpha-\varepsilon_i}_\sigma(i,X)|
  =\sum_{j=1}^m\beta_j|T^\alpha_{\sigma[j]}(i,X\cup\{n\})|.
  \]
  Therefore,
  \begin{align*}
    \sum_{i=1}^l(\alpha_i-1)\xi^{\alpha-\varepsilon_i}_\beta&=\sum_{i=1}^l\sum_{X\subseteq\mathbf n'}\sum_{j=1}^m\beta_j|T^\alpha_{\sigma[j]}(i,X\cup\{n\})|\\
    &=\sum_{j=1}^m\beta_j\sum_{i=1}^l\sum_{\substack{X\subseteq\mathbf n\\n\in X}}|T^\alpha_{\sigma[j]}(i,X)|\\
    &=\sum_{j=1}^m\beta_j|T^\alpha_{\sigma[j]}|.
  \end{align*}
  Using Lemma \ref{lemma:PermutationCharacterEqualsCardinality} and the fact that \(\xi^\alpha\) is a class function we get \[
  |T^\alpha_{\sigma[j]}|=\xi^\alpha(\sigma[j])=\xi^\alpha(\sigma_{\beta+\varepsilon_j})=\xi^\alpha_{\beta+\varepsilon_j}
  \]
  for every \(1\le j\le m\).  This completes the proof.
\end{proof}

\section{An adjoint pair}
\label{sec:AnAdjointPair}

In this section, we reformulate the reciprocity relationship of the preceding section, expressing it as an adjoint relationship between a pair of \(\Z\)-linear maps.

Denote by \(A\) the free \(\Z\)-module on the set \(\{x^\alpha\mid\alpha\in\Gamma\}\) and denote by \(B\) the free \(\Z\)-module on the set \(\{x_\beta\mid\beta\in\Gamma^+\}\).

Denote by \((\,\cdot\,,\,\cdot\,):A\times B\to\Z\) the \(\Z\)-bilinear map uniquely determined by
\[
(x^\alpha,x_\beta)=\xi^\alpha_\beta\quad(\alpha\in\Gamma, \beta\in\Gamma^+).
\]

Let \(l,m\in\N\).  Denote by \(\delta^-_l:A\to A\) and \(\delta^+_m:B\to B\) the \(\Z\)-linear maps uniquely determined by
\[
\delta^-_l(x^\alpha)=\sum_{i=1}^l(\alpha_i-1)x^{\alpha-\varepsilon_i}\quad(\alpha\in\Gamma)
\]
and
\[
\delta^+_m(x_\beta)=\sum_{j=1}^m\beta_jx_{\beta+\varepsilon_j}\quad(\beta\in\Gamma^+),
\]
respectively.

Put \(A_l=\langle x^\alpha\mid\alpha\in\Gamma_l\rangle\le A\) and \(B_m=\langle x_\beta\mid\beta\in\Gamma^+_m\rangle\le B\).

\begin{theorem}
\label{thm:AdjointPair}
  For every \(a\in A_l\) and \(b\in B_m\), we have
  \[
  (\delta^-_l(a),b)=(a,\delta^+_m(b)).
  \]
\end{theorem}
\begin{proof}
  Let \(a\in A_l\) and \(b\in B_m\).  We assume, without loss of generality (due to linearity), that \(a=x^\alpha\) and \(b=x_\beta\) with \(\alpha\in\Gamma_l\) and \(\beta\in\Gamma^+_m\).   Using Theorem \ref{thm:ReciprocalPermutation}, we have
  \begin{align*}
  (\delta^-_l(a),b)&=(\sum_{i=1}^l(\alpha_i-1)x^{\alpha-\varepsilon_i},x_\beta)
  =\sum_{i=1}^l(\alpha_i-1)\xi^{\alpha-\varepsilon_i}_\beta\\
  &=\sum_{j=1}^m\beta_j\xi^\alpha_{\beta+\varepsilon_j}
  =(x^\alpha,\sum_{j=1}^m\beta_jx_{\beta+\varepsilon_j})\\
  &=(a,\delta^+_m(b)).
  \end{align*}
\end{proof}

\section{Reciprocity for irreducible characters}
\label{sec:ReciprocityForIrreducibleCharacters}

The goal of this section is Theorem \ref{thm:ReciprocalIrreducible}, which provides a reciprocal relationship for the irreducible character values for \(S_n\) (\(n\in\N\)) analogous to the relationship given in Theorem \ref{thm:ReciprocalPermutation}.

Let \(n\in\N\).  For \(\alpha\vdash n\), denote by \(\zeta^\alpha\) the irreducible character of \(S_n\) corresponding to \(\alpha\) \cite[2.2.5, p.~39]{MR644144}, and for \(\beta\in\Gamma^+\) with \(|\beta|=n\), put \(\zeta^\alpha_\beta=\zeta^\alpha(\sigma_\beta)\).  The matrix \([\zeta^\alpha_\beta]_{\alpha,\beta}\) with \(\alpha,\beta\vdash n\) (relative to a choice of ordering of the partitions) is the character table of \(S_n\).

We write \(\sgn\) for the sign character of \(S_n\), so \(\sgn(\sigma)\) is \(1\) or \(-1\) according as the permutation \(\sigma\) is even or odd.

For \(l\in\N\),  put \(\id_l=(1,2,3,\dots,l,0,0,\dots\,)\in\Gamma\).

For \(\alpha\in\Gamma\), we have \(\alpha\in\Gamma_l\) for some \(l\in\N\); put
\[
\chi^\alpha=\sum_{\sigma\in S_l}\sgn(\sigma)\xi^{\alpha+\sigma-\id_l}.
\]
(As observed in \cite[p.~47]{MR644144}, this definition is independent of the choice of \(l\).) In the sum, the permutation \(\sigma\in S_l\) is regarded as the element \((\sigma(1),\sigma(2),\dots,\sigma(l),0,0,\dots\,)\) of the group \(\Gamma\).

\begin{theorem}[{\cite[2.3.15, p.~52]{MR644144}}]
\label{thm:ChiEqualsIrreducible}
For every \(\alpha\vdash n\), we have \(\zeta^\alpha=\chi^\alpha.\)
\qed
\end{theorem}

For \(l\in\N\), denote by \(D_l:A\to A\) the \(\Z\)-linear map uniquely determined by
\[
D_l(x^\alpha)=\sum_{\sigma\in S_l}\sgn(\sigma) x^{\alpha+\sigma-1_l}\quad(\alpha\in\Gamma),
\]
where \(1_l=(1,1,\dots,\underset l1,0,0,\dots\,)\in\Gamma\).

For \(\alpha\in\Gamma\) and \(\beta\in\Gamma^+\) with \(|\alpha|=|\beta|\), put \(\chi^\alpha_\beta=\chi^\alpha(\sigma_\beta)\).

\begin{lemma}
\label{lemma:ChiUsingBilinearMap}
Let \(l\in\N\), \(\alpha\in\Gamma_l\), and \(\beta\in\Gamma^+\), with \(|\alpha|=|\beta|\). We have
\[
\chi^\alpha_\beta=(D_l(x^{\alpha-\id_l+1_l}),x_\beta).
\]
\end{lemma}
\begin{proof}
  Using the definitions, we have
  \begin{align*}
    \chi^\alpha_\beta&=\sum_{\sigma\in S_l}\sgn(\sigma)\xi^{\alpha+\sigma-\id_l}_\beta
    =\sum_{\sigma\in S_l}\sgn(\sigma)(x^{\alpha+\sigma-\id_l},x_\beta)\\
    &=(\sum_{\sigma\in S_l}\sgn(\sigma)x^{\alpha+\sigma-\id_l},x_\beta)
    =(D_l(x^{\alpha-\id_l+1_l}),x_\beta).
  \end{align*}
\end{proof}

\begin{lemma}
\label{lemma:DeterminantCommutesWithDerivation}
  For each \(l\in\N\), we have \(D_l\delta^-_l=\delta^-_lD_l\).
\end{lemma}
\begin{proof}
  Let \(l\in\N\) and \(\alpha\in\Gamma\).  On the one hand,
  \begin{align*}
    D_l\delta^-_l(x^\alpha)&=D_l\Big(\sum_{i=1}^l(\alpha_i-1)x^{\alpha-\varepsilon_i}\Big)\\
    &=\sum_{i=1}^l(\alpha_i-1)\sum_{\sigma\in S_l}\sgn(\sigma)x^{\alpha-\varepsilon_i+\sigma-1_l}.
  \end{align*}
  On the other hand,
  \begin{align*}
    \delta^-_lD_l(x^\alpha)&=\delta^-_l\big(\sum_{\sigma\in S_l}\sgn(\sigma)x^{\alpha+\sigma-1_l}\big)\\
    &=\sum_{\sigma\in S_l}\sgn(\sigma)\sum_{i=1}^l(\alpha_i+\sigma(i)-2)x^{\alpha+\sigma-1_l-\varepsilon_i}.
  \end{align*}
  Therefore,
  \begin{align*}
    (\delta^-_lD_l-D_l\delta^-_l)(x^\alpha)&=\sum_{\sigma\in S_l}\sgn(\sigma)\sum_{i=1}^l(\sigma(i)-1)x^{\alpha+\sigma-1_l-\varepsilon_i}\\
    &=\sum_{\sigma\in S_l}\sgn(\sigma)\sum_{i=1}^l(i-1)x^{\alpha+\sigma-1_l-\varepsilon_{\sigma^{-1}(i)}}\\
    &=\sum_{i=1}^l(i-1)\sum_{\sigma\in S_l}\sgn(\sigma)x^{\alpha+\sigma-1_l-\varepsilon_{\sigma^{-1}(i)}}.
  \end{align*}
Let \(1<i\le l\).  Put \(\tau=(i-1,i)\in S_l\), and for \(\sigma\in S_l\), put
\[
\gamma^\sigma=\alpha+\sigma-1_l-\varepsilon_{\sigma^{-1}(i)}.
\]
Let \(\sigma\in S_l\).  For \(1\le k\le l\), we have
\begin{align*}
\gamma^\sigma_k&=
\begin{cases}
  \alpha_k+i-2,&\text{if \(\sigma(k)\in\{i-1,i\}\)},\\
  \alpha_k+\sigma(k)-1,&\text{otherwise},
\end{cases}
\\
&=\gamma^{\tau\sigma}_k,
\end{align*}
and for \(k>l\), we have \(\gamma^\sigma_k=\alpha_k=\gamma^{\tau\sigma}_k\).  Therefore, \(\gamma^\sigma=\gamma^{\tau\sigma}\), giving
\[
\sum_{\sigma\in S_l}\sgn(\sigma)x^{\alpha+\sigma-1_l-\varepsilon_{\sigma^{-1}(i)}}=\sum_{\substack{\sigma\in S_l\\ \sigma\text{, even}}}(x^{\gamma^\sigma}-x^{\gamma^{\tau\sigma}})=0.
\]

We conclude that \(D_l\delta^-_l=\delta^-_lD_l\).
\end{proof}

\begin{lemma}[{\cite[2.3.9, p.~48]{MR644144}}]
\label{lemma:SwitchEntries}
For every \(\gamma\in\Gamma\) and \(i\in\Z^+\), we have \(\chi^\gamma=-\chi^\eta\), where
\[
\eta=(\gamma_1,\dots,\gamma_{i-1},\gamma_{i+1}-1,\gamma_i+1,\gamma_{i+2},\dots\,).
\]
\qed
\end{lemma}

The formula in the following theorem is analogous to that in Theorem \ref{thm:ReciprocalPermutation}, but we point out (lest it be supposed a misprint) that the factor here is \(\alpha_i-i\) instead of \(\alpha_i-1\) as earlier.

\begin{theorem}
\label{thm:ReciprocalIrreducible}
  Let \(n\in\Z^+\).  For each \(\alpha\vdash n\) and each \(\beta\in\Gamma^+\) with \(|\beta|=n-1\), we have
  \[
  \sum_{\substack{i=1\\\alpha_i>\alpha_{i+1}}}^{l(\alpha)}(\alpha_i-i)\zeta^{\alpha-\varepsilon_i}_\beta=\sum_{j=1}^{l(\beta)}\beta_j\zeta^\alpha_{\beta+\varepsilon_j}.
  \]
\end{theorem}
\begin{proof}
Let \(\alpha\vdash n\) and let \(\beta\in\Gamma^+\) with \(|\beta|=n-1\).

Fix \(1\le i\le l(\alpha)\).  We first observe that if \(\alpha_i>\alpha_{i+1}\), then \(\alpha-\varepsilon_i\vdash(n-1)\), so \(\zeta^{\alpha-\varepsilon_i}_\beta\) is defined and equals \(\chi^{\alpha-\varepsilon_i}_\beta\) by Theorem \ref{thm:ChiEqualsIrreducible}.  Now assume that \(\alpha_i\ngtr\alpha_{i+1}\). Then \(\alpha_i=\alpha_{i+1}\), implying \(\gamma_i=\gamma_{i+1}-1\), where \(\gamma=\alpha-\varepsilon_i\).  Then, in the notation of Lemma \ref{lemma:SwitchEntries}, we have \(\eta=\gamma\).  Therefore, by that lemma, \(\chi^\gamma=-\chi^\gamma\), implying \(\chi^{\alpha-\varepsilon_i}=\chi^\gamma=0\).

Put \(l=l(\alpha)\) and \(m=l(\beta)\), and denote by LHS the left-hand side of the equation in the statement.  Using the preceding paragraph for the first equality, and then Lemma \ref{lemma:ChiUsingBilinearMap} we have
\begin{align*}
  \text{LHS}
  &=\sum_{i=1}^l(\alpha_i-i)\chi^{\alpha-\varepsilon_i}_\beta
  =\sum_{i=1}^l(\alpha_i-i)(D_l\big(x^{\alpha-\varepsilon_i-\id_l+1_l}\big),x_\beta)\\
  &=(D_l\bigg(\sum_{i=1}^l(\alpha_i-i)x^{\alpha-\varepsilon_i-\id_l+1_l}\bigg),x_\beta)
  =(D_l\big(\delta^-_l(x^{\alpha-\id_l+1_l})\big),x_\beta)\\
  &=(\delta^-_l\big(D_l(x^{\alpha-\id_l+1_l})\big),x_\beta),
\end{align*}
where the last equality uses Lemma \ref{lemma:DeterminantCommutesWithDerivation}. Next, the argument of \(\delta^-_l\) is seen to be in \(A_l\), and \(x_\beta\in B_m\), so Theorem \ref{thm:AdjointPair} applies and we get
\begin{align*}
  \text{LHS}
  &=(D_l\big(x^{\alpha-\id_l+1_l}\big),\delta^+_m(x_\beta))
  =(D_l\big(x^{\alpha-\id_l+1_l}\big),\sum_{j=1}^m\beta_j(x_{\beta+\varepsilon_j}))\\
  &=\sum_{j=1}^m\beta_j(D_l\big(x^{\alpha-\id_l+1_l}\big),x_{\beta+\varepsilon_j})
  =\sum_{j=1}^m\beta_j\chi^\alpha_{\beta+\varepsilon_j}\\
  &=\sum_{j=1}^{l(\beta)}\beta_j\zeta^\alpha_{\beta+\varepsilon_j},
\end{align*}
again using Lemma \ref{lemma:ChiUsingBilinearMap} and then Theorem \ref{thm:ChiEqualsIrreducible}.
\end{proof}

\section{Recursion formula}
\label{sec:RecursionFormula}

The theorem below provides a method for recursively finding the irreducible characters values for \(S_n\) (\(n\in\Z^+\)).  Each value \(\zeta^\alpha_\beta\) is expressed in terms of values \(\zeta^\gamma_\delta\) with \(\gamma\vdash n-1\) or with \(\gamma=\alpha\) and \(\delta<\beta\), where \(<\) is the reverse lexicographical order on the set of partitions of \(n\) (i.e., \(\delta<\beta\) if for some \(k\in\Z^+\) we have \(\delta_j=\beta_j\) for \(j<k\) and \(\delta_k>\beta_k\)).

Therefore, if the character table for \(S_{n-1}\) is known, then the character table for \(S_n\) can be determined by taking the irreducible characters in turn (in any order) and working through the conjugacy classes of \(S_n\) ordered using the reverse lexicographic ordering of the associated partitions.

The first case in the theorem (\(\beta_m=1\)) is the case where the permutation at which the character is being evaluated has a fixed point, so the character value is given by the branching theorem \cite[2.4.3, p.~59]{MR644144}.  We have included the formula in this case for the sake of completeness.

The statement of the second case (\(\beta_m\not=1\)) requires additional notation:  For \(\gamma\in\Gamma\), put \(\mu(\gamma)=(\mu_1,\mu_2,\dots,\mu_m)\), where \(m=l(\gamma)\) and
\[
\mu_j=|\{1\le k\le m\mid\gamma_k=\gamma_j\}|\quad(1\le j\le m).
\]

The sole irreducible character value \(\zeta^0_0=1\) for \(S_0\) begins the recursion.

\begin{theorem}
\label{thm:RecursionFormula}
  Let \(n\in\Z^+\), let \(\alpha,\beta\vdash n\), and put \(m=l(\beta)\).
  \begin{itemize}
  \item[(i)] If \(\beta_m=1\), then
  \[
  \zeta^\alpha_\beta=\sum_{\substack{i=1\\\alpha_i>\alpha_{i+1}\\\,}}^{l(\alpha)}\zeta^{\alpha-\varepsilon_i}_{\beta-\varepsilon_m}\quad\text{\normalfont (the branching theorem)}.
  \]
  \item[(ii)] If \(\beta_m\not=1\), then
  \[
  \zeta^\alpha_\beta=(\beta_m-1)^{-1}\Bigg[\sum_{\substack{i=1\\\alpha_i>\alpha_{i+1}\\\,}}^{l(\alpha)}(\alpha_i-i)\zeta^{\alpha-\varepsilon_i}_{\beta-\varepsilon_m}-\sum_{\substack{j=1\\\beta_j<\beta_{j-1}}}^{m-1}\mu_j\beta_j\zeta^\alpha_{\beta+\varepsilon_j-\varepsilon_m}
  \Bigg],
  \]
  where \(\mu=\mu(\beta-\varepsilon_m)\) and \(\beta_0:=\infty\).
  \end{itemize}
\end{theorem}
\begin{proof}
  (i) See \cite[pp.~58--59]{MR644144}.

  (ii) Assume that \(\beta_m\not=1\).  By Theorem \ref{thm:ReciprocalIrreducible}, we have
  \begin{equation}
  \label{eqn:RecursionFirst}
  \sum_{\substack{i=1\\\alpha_i>\alpha_{i+1}}}^{l(\alpha)}(\alpha_i-i)\zeta^{\alpha-\varepsilon_i}_{\beta-\varepsilon_m}
  =\sum_{j=1}^{m-1}\beta_j\zeta^\alpha_{\beta-\varepsilon_m+\varepsilon_j}+(\beta_m-1)\zeta^\alpha_\beta.
  \end{equation}
  If \(\beta_k=\beta_j\) for some \(1\le k,j\le m-1\), then the permutations \(\sigma_{\beta+\varepsilon_k-\varepsilon_m}\) and \(\sigma_{\beta+\varepsilon_j-\varepsilon_m}\) are conjugate, implying \(\zeta^\alpha_{\beta+\varepsilon_k-\varepsilon_m}=\zeta^\alpha_{\beta+\varepsilon_j-\varepsilon_m}\).  Therefore,
  \begin{equation}
  \label{eqn:RecursionSecond}
  \begin{aligned}
  \sum^{m-1}_{j=1}\beta_j\zeta^\alpha_{\beta-\varepsilon_m+\varepsilon_j}&=\sum_{\substack{j=1\\\beta_j<\beta_{j-1}}}^{m-1}\sum_{\substack{k=1\\\beta_k=\beta_j}}^{m-1}\beta_k\zeta^\alpha_{\beta+\varepsilon_k-\varepsilon_m}\\
  &=\sum_{\substack{j=1\\\beta_j<\beta_{j-1}}}^{m-1}|M_j|\beta_j\zeta^\alpha_{\beta+\varepsilon_j-\varepsilon_m},
  \end{aligned}
  \end{equation}
  where \(M_j=\{1\le k\le m-1\mid\beta_k=\beta_j\}\).

  Fix \(1\le j\le m-1\) with \(\beta_j<\beta_{j-1}\).  For \(1\le k\le m-1\), we have \(\beta_k=(\beta-\varepsilon_m)_k\) and \(\beta_j=(\beta-\varepsilon_m)_j\), while
  \[
  (\beta-\varepsilon_m)_m=\beta_m-1<\beta_j=(\beta-\varepsilon_m)_j.
  \]
  Therefore,
  \[
  M_j=\{1\le k\le m\mid(\beta-\varepsilon_m)_k=(\beta-\varepsilon_m)_j\},
  \]
  which gives \(|M_j|=\mu_j\).

  We now get the formula in the statement by substituting this into Equation (\ref*{eqn:RecursionSecond}), substituting that result into Equation (\ref*{eqn:RecursionFirst}), and finally solving for \(\zeta^\alpha_\beta\).
\end{proof}

\begin{remark}
We can sacrifice readability a bit for the sake of compactness and express both cases in the theorem using a single formula by using the Kronecker delta:  Putting \(\kappa=1-\delta_{\beta_m1}\), we have
\[
\zeta^\alpha_\beta=(\beta_m-\kappa)^{-1}\Bigg[\sum_{\substack{i=1\\\alpha_i>\alpha_{i+1}\\\,}}^{l(\alpha)}(\alpha_i-i)^\kappa\zeta^{\alpha-\varepsilon_i}_{\beta-\varepsilon_m}-\kappa\sum_{\substack{j=1\\\beta_j<\beta_{j-1}}}^{m-1}\mu_j\beta_j\zeta^\alpha_{\beta+\varepsilon_j-\varepsilon_m}\Bigg].
\]
\qed
\end{remark}

As an application we give a proof of the well-known formula for the value of an irreducible character of \(S_n\) at an \(n\)-cycle.  In the proof, we use the term ``hook,'' which refers to a partition of \(n\) of the form
\[
(n-r,1^r):=(n-r,1,1,\dots,1,0,0,\dots\,)
\]
with \(r\) ones.

\begin{corollary}[{\cite[2.3.17, p.~54]{MR644144}}]
\label{cor:IrreducibleCharacterAtCycle}
  Let \(n\in\Z^+\), let \(\alpha\vdash n\), and put \(\beta=(n,0,0,\dots)\). We have
  \[
  \zeta^\alpha_\beta=
  \begin{cases}
    (-1)^r,&\text{if \(\alpha=(n-r,1^r)\), some \(0\le r<n\)},\\
    0,&\text{otherwise}.
  \end{cases}
  \]
\end{corollary}
\begin{proof}
  We proceed by induction on \(n\).  If \(l(\alpha)=1\), then \(\alpha=(n)\) and \(\zeta^\alpha\) is the trivial character, so the claim holds with \(r=0\).  Now assume that \(l(\alpha)>1\).  In particular, \(n>1\).

  We have \(m=l(\beta)=1\) and \(\beta_m=n\not=1\), so part (ii) of Theorem \ref{thm:RecursionFormula} applies.  In that formula the sum over \(j\) is empty, while, for each index \(i\) in the first sum, the induction hypothesis applies, since \(\alpha-\varepsilon_i\vdash(n-1)\) and \(\beta-\varepsilon_m=(n-1,0,0,\dots)\), giving  \(\zeta^{\alpha-\varepsilon_i}_{\beta-\varepsilon_m}=0\) if \(\alpha-\varepsilon_i\) is not a hook.

  Assume that \(\alpha\) is not a hook and assume that \(\alpha-\varepsilon_i\) is a hook for some \(1\le i\le l(\alpha)\) with \(\alpha_i>\alpha_{i+1}\).  Then \(i=2\) and \(\alpha_i=2\), so that \(\alpha_i-i=0\).  Therefore, the theorem gives \(\zeta^\alpha_\beta=0\), which establishes the second case.

  Now assume that \(\alpha=(n-r,1^r)\) for some \(0\le r<n\).  Since \(l(\alpha)>1\), we have \(r>0\).  If \(r=n-1\), then \(\alpha=(1^n)\) and \(\zeta^\alpha\) is the alternating character, so the claim holds.  Now assume that \(r<n-1\), so that \(n-r>1\).  Then the formula in the theorem has two terms, corresponding to \(i=1\) and \(i=r+1\), respectively, and we get
  \begin{align*}
  \zeta^\alpha_\beta&=(\beta_1-1)^{-1}\left[(\alpha_1-1)\zeta^{\alpha-\varepsilon_1}_{\beta-\varepsilon_1}+(\alpha_{r+1}-(r+1))\zeta^{\alpha-\varepsilon_{r+1}}_{\beta-\varepsilon_1}\right]\\
  &=(n-1)^{-1}\left[(n-r-1)(-1)^r+(1-(r+1))(-1)^{(r-1)}\right]\\
  &=(-1)^r,
  \end{align*}
  again using the induction hypothesis.  This establishes the first case and completes the proof.
\end{proof}

We chose to streamline this proof by using that the cases \(\alpha=(n)\) and \(\alpha=(1^n)\) correspond to the trivial character and the alternating character, respectively, but this was not necessary since Theorem \ref{thm:RecursionFormula} handles these cases as well.

\bibliographystyle{amsalpha}
\bibliography{HolmesBib}

\end{document}